\documentclass[a4paper]{article}

\usepackage[backend=bibtex,giveninits=true]{biblatex}
\usepackage{amsmath,amssymb,amsthm}
\usepackage{comment}
\usepackage{bbold}
\usepackage{graphicx}
\usepackage{tikz}
\usetikzlibrary{shapes,arrows}
\usepackage{verbatim}
\usepackage{hyperref}

\newcommand{\Z}{\mathbb{Z}}
\newcommand{\N}{\mathbb{N}}
\newcommand{\R}{\mathbb{R}}

\newcommand{\meas}{\mathcal{M}}
\newcommand{\action}{\curvearrowright}
\newcommand{\unifd}[1]{#1^{\Z^d}}
\newcommand{\unif}[1]{#1^{\Z^2}}
\newcommand{\stR}[3]{\draw (#1,#2) rectangle +(1,1) node[midway] {#3};}

\DeclareMathOperator{\chull}{CHull}

\DeclareMathOperator{\supp}{Supp}

\theoremstyle{plain}
\newtheorem{theorem}{Theorem}[section]
\newtheorem{lemma}[theorem]{Lemma}

\theoremstyle{definition}
\newtheorem{definition}[theorem]{Definition}
\newtheorem{example}[theorem]{Example}
\newtheorem{question}[theorem]{Question}

\theoremstyle{remark}
\newtheorem{claim}[theorem]{Claim}

\title{Cellular Automata and Bootstrap Percolation}

\author{Ville Salo$^1$ \and Guillaume Theyssier$^2$ \and Ilkka Törmä$^1$}

\date{%
  $^1$ Department of Mathematics and Statistics \\
  University of Turku \\
  Turku, Finland \\
  \texttt{vosalo@utu.fi}, \texttt{iatorm@utu.fi} \\
  \phantom{x} \\
  $^2$ Aix Marseille University, CNRS, I2M\\
  Marseille, France \\
  \texttt{guillaume.theyssier@cnrs.fr} \\
  \phantom{x} \\
  \today}

\begin{document}
\maketitle

\begin{abstract}
  We study qualitative properties of two-dimensional freezing cellular automata with a binary state set initialized on a random configuration.
  If the automaton is also monotone, the setting is equivalent to bootstrap percolation.
  We explore the extent to which monotonicity constrains the possible asymptotic dynamics by proving two results that do not hold in the subclass of monotone automata.
  First, it is undecidable whether the automaton almost surely fills the space when initialized on a Bernoulli random configuration with density $p$, for some/all $0 < p < 1$.
  Second, there exists an automaton whose space-filling property depends on $p$ in a non-monotone way.
\end{abstract}

\section{Introduction}

This paper is motivated by two a priori distinct research domains: bootstrap percolation on one hand, where a complete and explicit classification of possible behaviors is sought \cite{BaBoPrSm16,DeLaDa16,BoSmUz15}, and theory of general cellular automata initialized on random configurations on the other hand, where undecidability is the norm and possible behaviors are sometimes only constrained by computability considerations \cite{HELLOUIN_DE_MENIBUS_2016,Boyer_2015}. 
We address two problems at the interface of these two domains:
\begin{itemize}
\item what are the key ingredients that make it possible to decide the
  global qualitative behavior of a model from its local definition in
  the bootstrap percolation settings, while this is highly unsolvable
  in general for cellular automata?
\item what kind of new behaviors can appear when relaxing the standard settings of bootstrap percolation?
\end{itemize}

Bootstrap percolation is a class of deterministic growth models in random environments.
The basic premise is that we have a discrete universe of \emph{sites}, typically arranged on a regular lattice such as $\Z^d$, a random subset of which are initially \emph{infected}, typically drawn from a Bernoulli distribution with some density $0 < p < 1$.
A deterministic rule, typically uniform in space and time, allows the infection to spread into healthy sites that have enough infected neighbors.
The main quantities of interest are then the probability of every site being eventually infected (called \emph{percolation}), and the distribution of the time of infection of a given site, as a function of the density $p$.
Bootstrap percolation was introduced by Chalupa, Leath and Reich in~\cite{ChLeRe79} as a model of impurities in magnetic materials.
See~\cite{DeLaDa16,morris_2017,Morris2017} for an overview of subsequent literature.

Bootstrap percolation processes can be formalized as cellular automata (CA for short) on the binary state set $\{0,1\}$ that are \emph{monotone} ($x \leq y$ implies $f(x) \leq f(y)$) and \emph{freezing} ($x \leq f(x)$ always holds) with respect to the cellwise partial order.
The automaton is initialized on a random configuration $x \in \{0,1\}^{\Z^d}$, and the freezing property guarantees that the iterates $f^n(x)$ converge to a limit configuration.
Percolation corresponds to this limit being the all-1 configuration.
We say that $f$ trivializes the initial probability measure, if percolation happens almost surely.
Percolation properties of the entire class of binary monotone freezing cellular automata were explicitly studied in~\cite{BoSmUz15,BaBoPrSm16}.
In particular, the articles contain a characterization of the classes of rules that percolate almost surely for all $p$ or any $p$, and show that they are algorithmically decidable.

Interestingly, if we study general cellular automata initialized on random configurations, a lot of algorithmic undecidability arises. For instance, $\mu$-nilpotency, which can be seen as an equivalent of trivialization of the initial probability measure (see Lemma~\ref{lem:percolationasmunilpotency}), becomes a $\Pi_3^0$-complete property \cite[Theorem 5.7]{Boyer_2015}. Also examples can be constructed, such that the limit behavior depends on the density of some state in the initial probability measure in a complex way \cite[Theorem 3]{HELLOUIN_DE_MENIBUS_2016}.
The constructions behind these results are not monotone and use non-convergent orbits, so they clearly break both the monotone and freezing hypothesis of bootstrap percolation models.

In this article we study the variety of asymptotic behaviors exhibited by freezing CA when initialized on Bernoulli random configurations, contrasted to bootstrap percolation models.
In other words, we investigate the role of monotonicity in the qualitative theory of bootstrap percolation.
We show that dropping the monotonicity requirement results in a richer set of possible asymptotic behaviors.
Such automata may still be understood as models of physical or sociological phenomena.
For example, if the cells of a graph represent agents with political leanings, then non-monotone rules can model individuals becoming suspicious of a sudden influx of opposing views among their peers.
Examples of freezing non-monotone CA have been considered in the literature, like the ``rule one'' of S. Ulam~\cite{ulam} as an attempt to study models of crystal growth, or ``life without death''~\cite{GriMoo96} which is a freezing version of Conway's Life.
The dynamics of freezing cellular automata have been studied explicitly in e.g.~\cite{GoOlTh15,GoMaMoOl17,BeMaOlTh18}.
We note that in the literature it is common to require freezing CA to be decreasing rather than increasing, but here we choose to follow the opposite convention of percolation theory.

Some of the results on monotone CA extend to the context of freezing CA, such as the well known fact that trivialization of at least one nontrivial Bernoulli measure is equivalent to the existence of a fixed point with a nonzero but finite number of 0-cells (see Lemma~\ref{lem:CofiniteIce}).
As our first main result, we show that the property is not decidable in the class of binary feezing CA, and neither is the dual property of trivialization of all nontrivial Bernoulli measures.
The latter property is more interesting, since it is not obviously equivalent to any simple combinatorial condition.
It was shown in~\cite{BaBoPrSm16} that not all freezing monotone CA fall into one of these classes, as some exhibit a nontrivial phase transition at some critical probability $0 < p_c < 1$, percolating almost never for initial density $p < p_c$ and almost surely for $p > p_c$.
Our second main result is the construction of a freezing CA that has several such phase transitions: it trivializes the Bernoulli measure of density $p$ but not the one of density $q$, for some $0 < p < q < 1$.
In particular, the trivialization property is not monotone with respect to the initial density.

Several open problems arise naturally from our investigation.
First, in the context of cellular automata it is natural to ask whether the results extend to arbitrary finite state sets.
We prove some of our auxiliary results in this context, but our main results concern the binary case.
Do monotone freezing CA with three or more states have significantly more complex dynamics than binary CA?
In particular, are the analogous trivialization properties decidable?
Second, our example of a freezing CA with two phase transitions can likely be generalized to realize a wide range of exotic trivialization phenomena.

\section{Definitions}

For a finite alphabet $A$ and $d \geq 1$ (we will mostly be dealing with the case $d = 2$), the $d$-dimensional \emph{full shift} is the set $A^{\Z^d}$ equipped with the prodiscrete topology.
Elements of $A^{\Z^d}$ are called \emph{configurations}.
For $a \in A$, the \emph{$a$-uniform configuration} $x = \unifd{a} \in A^{\Z^d}$ is defined by $x_{\vec v} = a$ for all $\vec v \in \Z^d$. We say a configuration $x \in A^{\Z^2}$ is \emph{$a$-finite} if $|\{\vec v \in \Z^2 \;|\; x_{\vec v} \neq a\}| < \infty$.
We have an action $\sigma : \Z^d \action A^{\Z^d}$ of the additive group $\Z^d$ by homeomorphisms, called the \emph{shift action}, given by $\sigma_{\vec v}(x)_{\vec n} = x_{\vec n + \vec v}$.
If $A$ is a poset (partially ordered set), then we see $A^{\Z^d}$ as a poset with the cellwise order: $x \leq y$ means $x_{\vec v} \leq y_{\vec v}$ for all $\vec v \in \Z^d$.

A $d$-dimensional \emph{pattern} is a function $w \in A^D$ with $D \subset \Z^d$ finite.
The topology of $A^{\Z^d}$ is generated by \emph{cylinder sets} of the form $[w]_{\vec v} = \{ x \in A^{\Z^d} \;|\; \sigma_{\vec v}(x)|_D = w \}$ for a pattern $w \in A^D$.
If $\vec v$ is omitted, it is assumed to be $\vec 0$.
In a slight abuse of notation, each symbol $a \in A$ stands for the pattern $\vec 0 \mapsto a$ with domain $D = \{\vec 0\}$, so that $[a]_{\vec v} = \{ x \in A^{\Z^d} \;|\; x_{\vec v} = a \}$.
For a domain $D \subset \Z^d$, $x \in A^{\Z^d}$ and $a \in A$, we write $x|_D \equiv a$ for $x|_D = a^D$.

A \emph{cellular automaton} (CA for short) is a function $f : A^{\Z^d} \to A^{Z^d}$ defined by a finite \emph{neighborhood} $N \subset \Z^d$ and a \emph{local rule} $F : A^N \to A$ with $f(x)_{\vec v} = F(\sigma_{-\vec v}(x)|_N)$ for all ${\vec v\in\Z^d}$.
If $\|\vec v\|_\infty \leq r$ holds for all $\vec v \in N$, we say $r$ is a \emph{radius} for $f$.
By the Curtis-Hedlund-Lyndon theorem \cite{hedlund}, CA are exactly the continuous functions from $A^{\Z^d}$ to itself that commute with the shift action.

Denote by $\meas(A^{\Z^d})$ the set of Borel probability measures on $A^{\Z^d}$, and by $\meas_\sigma(A^{\Z^d})$ the $\sigma$-invariant ones (which satisfy $\mu(\sigma_{\vec v}(X)) = \mu(X)$ for all Borel sets $X$ and all ${\vec v\in\Z^d}$).
We equip $\meas(A^{\Z^d})$ with the weak-${*}$ topology, or convergence on cylinder sets.
The \emph{support} $\supp(\mu)$ of $\mu \in \meas(A^{\Z^d})$ is the unique smallest closed set $K \subset A^{\Z^d}$ with $\mu(K) = 1$.
We can apply a CA $f$ to a measure $\mu \in \meas(A^{\Z^d})$ by $f(\mu)(X) = \mu(f^{-1} X)$.

For a probability vector $\pi : A \to [0,1]$ (that satisfies $\sum_{a \in A} \pi(a) = 1$), the \emph{product measure} $\mu_\pi \in \meas_\sigma(A^{\Z^d})$ is the unique Borel measure with $\mu([P]) = \prod_{\vec v \in D} \pi(P_{\vec v})$ for all patterns $P \in A^D$.
If $A = \{0,1\}$ and $0 \leq p \leq 1$, then $\mu_p = \mu_\pi$ is the Bernoulli measure with $\pi(1) = p$.
For $x \in A^{\Z^d}$, denote by $\delta_x$ the unique measure with $\delta_x(\{x\}) = 1$.

The (closed) convex hull of a set $K \subset \R^d$ is denoted $\chull(K)$.
The notation $\forall^\infty x \in X$ means ``for all but finitely many $x \in X$'', and $\exists^\infty x \in X$ means ``there exist infinitely many $x \in X$''.

\section{Freezing, monotonicity and measures}


\begin{definition}
Let $P$ be a finite poset. A cellular automaton $f$ on $P^{\Z^d}$ is \emph{freezing}, if $x \leq f(x)$ for all $x \in P^{\Z^d}$. It is \emph{monotone}, if $f(x) \leq f(y)$ for all $x \leq y \in P^{\Z^d}$.
\end{definition}

In this paper, when considering freezing cellular automata on ${\{0,1\}^{\Z^d}}$, we always implicitly refer to the poset with elements $0$ and $1$ such that ${0 < 1}$.

\begin{lemma}
\label{lem:freezingsets}
A cellular automaton $f$ on $\{0,1\}^{\Z^d}$ is freezing and monotone if and only if there exists a finite family $E$ of finite subsets of $\Z^d \setminus \{\vec 0\}$ with the following property.
For all $x \in \{0,1\}^{\Z^d}$, we have $f(x)_{\vec 0} = 1$ if and only if $x_{\vec 0} = 1$ or there exists $N \in E$ with $x|_N \equiv 1$.
\end{lemma}

\begin{proof}
Given a freezing and monotone $f$, choose $E$ as the family of minimal subsets of $\Z^d \setminus \{\vec 0\}$ such that $x|_N \equiv 1$ implies $f(x)_{\vec 0} = 1$ (all of these sets are subsets of the neighborhood of $f$, so $E$ is finite). The other direction is clear.
\end{proof}

\begin{definition}
For a freezing monotone cellular automaton $f$ on $\{0,1\}^{\Z^d}$, we write $E(f)$ for the set $E$ chosen in the proof of Lemma~\ref{lem:freezingsets}. We also denote
\[ F(f) = \{ N \in E(f) \;|\; \vec 0 \notin \chull(N) \}, \]
and $G(f) = E(f) \setminus F(f)$. For a finite family $E$ of incomparable subsets of $\Z^d \setminus \{\vec 0\}$, we denote by $f_E$ the cellular automaton defined by $E(f_E) = E$. If $E = \{N\}$ is a singleton, we may also abuse notation and write $f_N$ for $f_E$.
\end{definition}

\begin{definition}
Let $\mu \in \meas(A^{\Z^d})$ be a measure. The \emph{$\mu$-limit set} of a cellular automaton $f$ on $A^{\Z^d}$ is
\[ \Omega_f^\mu = \overline{\bigcup_{\nu \in \mathcal{F}} \supp(\nu)}, \]
where $\mathcal{F}$ is the set of limit points of the sequence $(f^n(\mu))_{n \in \N}$.
\end{definition}

If $\mu \in \meas_\sigma(A^{\Z^d})$ is shift-invariant, then $\Omega_f^\mu$ is the set of configurations $x$ such that no pattern $w$ occurring in $x$ satisfies $\lim_n \mu(f^{-n}([w]_{\vec 0})) = 0$.
$\mu$-limit sets were first defined in~\cite{KuMa00} in the shift-invariant case using this characterization.

\begin{definition}
Let $\mu \in \meas(A^{\Z^d})$ and let $f$ be a cellular automaton on $A^{\Z^d}$. We say $f$ \emph{trivializes} $\mu$, if $|\Omega_f^\mu| = 1$.
\end{definition}

In cellular automata literature, a CA $f$ is called \emph{$\mu$-nilpotent} if $\Omega_f^\mu = \{x\}$ for some unary configuration $x \in A^{\Z^d}$.

If $P$ is a poset with a maximal element $m$, and $f$ is a freezing CA on $P^{\Z^d}$ that trivializes a full-support product measure, then the limit measure must of course be concentrated on the $m$-uniform point.



\begin{lemma}
  \label{lem:percolationasmunilpotency}
  Let $P$ be a finite poset, $\mu \in \meas_\sigma(P^{\Z^d})$ of full support and $f$ a freezing cellular automaton on $P^{\Z^d}$. The following conditions are equivalent:
\begin{itemize}
\item $f$ trivializes $\mu$
\item $\lim_{n} f^n(\mu) = \delta_x$ for some (unary) $x \in P^{\Z^d}$
\item for some $p \in P$ and $\mu$-almost every $x$, we have $f^n(x)_{\vec z} = p$ for all $\vec z \in \Z^d$ and all large enough $n \in \N$ (depending on $\vec z$).
\end{itemize}
\end{lemma}

\begin{proof}
  Suppose that $f$ trivializes $\mu$, so that $\Omega_f^\mu = \{x\}$ for some $x \in P^{\Z^d}$.
  Since $\mu$ is shift-invariant, for each $\vec v \in \Z^d$ we have
  \[
    \{\sigma_{\vec v}(x)\} = \Omega_f^{\sigma_{\vec v} \mu} = \Omega_f^\mu = \{x\}
  \]
  and hence there exists $p \in P$ such that $x = \unif{p}$ is unary.
  Thus, for each finite pattern $w$ containing an occurrence of some $q \in P \setminus \{p\}$ we have $\lim_n f^n \mu([w]) = 0$, and for each all-$p$ pattern $w'$ we have $\lim_n f^n \mu([w']) = 1$.
  This implies $\lim_n f^n \mu = \delta_x$, the second item.
  The converse is clear, so the first two items are equivalent.
  
  Denote by $m$ a maximal element of $P$.
  Since $f$ is freezing,
  \begin{equation}
    \label{eq:freeze}
    [m]_{\vec z} \subseteq f^{-n}([m]_{\vec z})
  \end{equation}
  for all $n \in \N$ and $\vec z \in \Z^d$.
  From~\eqref{eq:freeze} and the full support of $\mu$ it follows that no other state than $m$ can be chosen as $p$ in the third item.
  Let
  \[
    B = \{ x \in P^{\Z^d} \;|\; \forall \vec v \in \Z^d \  \forall^\infty n \in \N : f^n(x)_{\vec v} = m \}
  \]
  be the set of configurations that satisfy the condition of the third item.
  Consider ${E_{\vec z,n} = f^{-n}(P^{\Z^d} \setminus [m]_{\vec z})}$.
  From~\eqref{eq:freeze} we have ${E_{\vec z,n+1}\subseteq E_{\vec z,n}}$ for all $\vec z$ and $n$, and $B = P^{\Z^d} \setminus \bigcup_{\vec z \in \Z^d} \bigcap_{n \in \N} E_{\vec z,n}$.
  
  Suppose that $f$ does not trivialize $\mu$.
  Then there must be $\epsilon>0$ such that ${\mu(E_{\vec 0,n})\geq\epsilon}$ for all $n\geq 0$ (otherwise we would have ${\lim_n\mu(f^{-n}([m]))=1}$ since sets $E_{\vec 0,n}$ are decreasing, a contradiction).
  By continuity of $\mu$ from above, we deduce ${\mu(\bigcap_{n \in \N} E_{\vec 0,n})\geq\epsilon}$, so $\mu(B)\leq 1-\epsilon<1$.
  Therefore the third item does not hold.
  
  Suppose then that the third item does not hold, so that $\mu(B) < 1$.
  Since ${P^{\Z^d} \setminus B = \bigcup_{\vec z \in \Z^d} \bigcap_{n \in \N} E_{\vec z,n}}$ has positive measure, $\epsilon := \mu(\bigcap_{n \in \N} E_{\vec z,n}) > 0$ for some $\vec z$.
  For each $n$, we then have $\sum_{p \neq m} \mu(f^{-n}([p]_{\vec z})) = \mu(E_{\vec z,n}) \geq \epsilon$, and in particular $\mu(f^{-n}([p]_{\vec z})) \geq \epsilon/|P|$ for some $p \in P \setminus \{m\}$.
  For some $p$ this holds for infinitely many $n$, so some limit point $\nu$ of $(f^n \mu)_{n \in \N}$ satisfies $\nu(f^{-n}([p]_{\vec z}) \geq \epsilon/|P|$.
  Thus $f$ does not trivialize $\mu$.
  We have shown that the third item is equivalent to the first.
\end{proof}

We note that the first two items of Lemma~\ref{lem:percolationasmunilpotency} are equivalent even without the freezing hypothesis, and for the third item we only need the condition that some state $m \in P$ is \emph{persistent}, that is, $x_{\vec v} = m$ implies $f(x)_{\vec v} = m$. More dynamically (and generally) stated, it suffices that some configuration $x \in A^{\Z^d}$ is \emph{Lyapunov stable}, meaning
\[ \forall \epsilon > 0: \exists \delta > 0: \forall y \in A^{\Z^d}: \left( d(x, y) < \delta \implies \forall n \in \N: d(f^n(x), f^n(y)) < \epsilon \right), \]
where $d$ is any metric for the Cantor topology of $A^{\Z^d}$.

\begin{example}
  \label{ex:simpleca}
  Let $h = f_{\{(0,1), (1,1)\}}$, and let $x \in \{0,1\}^{\Z^2}$. Then $h^n(x)_{\vec 0} = 0$ for all $n \in \N$ if and only if there exists a path $(\vec z_i)_{i \in \N}$ in $\Z^2$ such that $\vec z_0 = \vec 0$, $\vec z_{i+1} \in \{\vec z_i + (0,1), \vec z_i + (1,1)\}$ and $x_{\vec z_i} = 0$ for all $i \geq 0$.
  Indeed, if such a path exists, then every cell in it will always have the state $0$, including the origin. On the other hand, if an infinite path does not exist, by K\H{o}nig's lemma there is a bound for the length of the paths. If the maximal length of a path starting from the origin in $x$ is $k$, then that in $h(x)$ is $k-1$. Inductively, we see that $h^k(x)_{\vec 0} = 1$.

  Take the Bernoulli measure $\mu_p \in \meas(\{0,1\}^{\Z^2})$ for ${0<p<1}$, and consider the probability $\theta(p)$ that in a $\mu_p$-random configuration $x \in \{0,1\}^{\Z^2}$ there exists an infinite path $(\vec z_i)_{i \in \N}$ of 0-states as above.
  This probability is clearly nonincreasing with respect to $p$, and the infimum of those $p$ for which it equals 0 is $1 - p_c$, where $p_c$ is the critical probability of nearest-neighbor oriented site percolation on $\N^2$; see~\cite[Section~12.8]{Gr13} or \cite{Durrett_1984} for discussion on the analogous bond percolation model.
  If $p > 1 - p_c$, then $h$ trivializes $\mu_p$, and if $p < 1 - p_c$, then it does not.
\end{example}

\begin{definition}
Let $P$ be a finite poset with maximal element $m$, and $f$ a freezing cellular automaton on $P^{\Z^d}$. An \emph{obstacle} is an $m$-finite configuration $x \in P^{\Z^d}$ such that $x \neq \unif{m}$ and $f(x) = x$.
\end{definition}

Of course, $f$ not admitting any obstacle is equivalent to the condition that for all $m$-finite $x \in P^{\Z^d}$, we have $f^n(x) = \unif{m}$ for large enough $n$.
The existence of an obstacle was called \emph{subcriticality} in~\cite{GrGr96}.
In~\cite{BoSmUz15} the term was given a different meaning.

\begin{lemma}
\label{lem:CofiniteIce}
Let $P$ be a finite poset with maximal element $m$, and $f$ a freezing cellular automaton on $P^{\Z^d}$. The following conditions are equivalent:
\begin{itemize}
\item $f$ trivializes some full support product measure.
\item $f$ trivializes some full support measure.
\item $f$ does not admit an obstacle.
\item For some $\epsilon > 0$, $f$ trivializes any product measure giving probability at least $1 - \epsilon$ to $m$.
\end{itemize}
\end{lemma}

\begin{proof}
  It is clear that the first item implies the second item and that the fourth implies the first.

  Suppose then that the third item does not hold, and $y \in P^{\Z^d}$ is an obstacle. Let $r$ be a radius for $f$ and let $p = y|_{[-k,k]^2}$, where $k = \max \{ | \vec v |_\infty \;|\; x_{\vec v} \neq m \} + r$. For any measure $\mu$ of full support, we have a positive probability that $p$ appears at the origin of a $\mu$-random configuration $x$. Due to freezing and the assumption on the radius, it follows by induction that $f^n(x)|_{[-k,k]^2} = p$ for all $n$, thus $f$ does not trivialize any full support measure. It follows that the second item implies the third.

  Suppose then that the third condition holds, and again let $r$ be the radius of $f$. It is a classical fact in percolation theory that for any connected unoriented graph $G$ with bounded degree, there exists an initial density $p > 0$ such that site percolation occurs in $G$ with probability $0$ (see e.g.\ Lemma~11 in~\cite{BoRi06}). Let $p$ be such a density for the graph with vertices $\Z^d$ and an edge between $\vec v$ and $\vec w$ whenever $|\vec v - \vec w|_\infty \leq r$. Now let $\mu$ be a product measure with any positive parameters such that the probability of $m$ is at least $1 - p$. If $x$ is $\mu$-random, then there is almost surely no infinite path from the origin to infinity. By K\H{o}nig's lemma this is equivalent existence of a finite set $D \subset \Z^d$ such that $\vec 0 \in D$ and $x_{\vec v} = m$ whenever $\vec v \in D$ and $\vec w \notin D$ for some $|\vec v - \vec w|_\infty \leq r$. Let $y \in P^{\Z^d}$ be the configuration satisfying $y|_D = x|_D$ and $y_{\vec v} = m$ for $\vec v \notin D$. Then $f^n(y) = m^{\Z^d}$ for large enough $m$, since $f$ admits no obstacles. Due to freezing and the assumption on the radius we have $f^n(x)_{\vec 0} = f^n(y)_{\vec 0} = m$.
\end{proof}

Again, this lemma is true in much higher generality. It suffices that some state is persistent, or that the CA admits at least one Lyapunov stable configuration. The lattice $\Z^d$ may also be replaced by any finitely-generated group.

\section{Trivialization for general freezing CA}
\label{sec:non-monotone}

The following definition is from~\cite{BoSmUz15}.

\begin{definition}
  Let $f$ be a freezing and monotone CA on $\{0,1\}^{\Z^2}$.
  For a unit vector $\vec v \in S^1$, write $x^{\vec v} \in \{0,1\}^{\Z^2}$ for the configuration defined by
  \[
    x^{\vec v}_{\vec z} =
    \begin{cases}
      1, & \text{if $\vec v \cdot \vec z < 0$,} \\
      0, & \text{otherwise.}
    \end{cases}
  \]
  A direction $\vec v \in S^1$ is \emph{stable} for $f$, if $f(x^{\vec v}) = x^{\vec v}$.
  The set of stable directions for $f$ is denoted $S(f)$.
\end{definition}

By~\cite[Theorem~1.10]{BoSmUz15}, $S(f)$ is a finite union of closed sub-intervals of $S^1$ with rational endpoints, and it is computable from the local rule of $f$.
The other results of \cite{BoSmUz15} and~\cite[Theorem~1]{BaBoPrSm16} together imply that the conditions of Lemma~\ref{lem:CofiniteIce} are equivalent to $S(f) \neq S^1$, and that $f$ trivializes all nontrivial Bernoulli measures if and only if there exists an open semicircle disjoint from the interior of $S(f)$.
Thus it is decidable whether a freezing monotone CA trivializes at least one nontrivial Bernoulli measure, or all of them.
We now show that when the monotonicity requirement is dropped, both problems become undecidable.


\begin{theorem}
  \label{thm:undecidable}
  The problem of whether a given binary freezing CA trivializes some nontrivial Bernoulli measure is undecidable.
  The same holds for trivializing all nontrivial Bernoulli measures.
\end{theorem}

We use some basic notions from automata theory and tiling theory in the construction.
A \emph{Turing machine} $T$ is defined by a finite set $Q$ of internal states with distinguished initial state $q_0 \in Q$ and halting state $q_h \in Q$, a finite set $\Sigma$ of tape symbols with a distinguished blank symbol $B \in \Sigma$, and a transition function $\delta : Q \times \Sigma \to Q \times \Sigma \times \{{+1}, {-1}\}$.
A value $\delta(p, s) = (q, r, d)$ indicates that when the read-write head of $T$ is in state $p$ on a tape symbol $s$, it should assume state $q$, write $r$ on the tape, and take one step in the direction $d$.
The machine is initialized in state $q_0$ on the left end of a right-infinite tape filled with $B$-symbols.
We assume that it never tries to step to the left of the leftmost tape cell, and will only enter the state $q_h$ on the leftmost tape cell.
These assumptions do not affect the undecidability of the \emph{halting problem} of whether the machine eventually enters the state $q_h$.

Let $C$ be a finite set of colors.
A \emph{Wang tile set on $C$} is a set of quadruples $S \subset C^4$.
A Wang tile $(a,b,c,d)$ should be visualized as a square whose east, north, west and south edges are labeled by the four colors in this order.
A pattern $P \in S^D$ over $S$ and of domain ${D\subseteq\Z^2}$ is \emph{valid} if the east color of $P_{(i,j)}$ equals the west color of $P_{(i+1,j)}$ whenever $(i,j), (i+1,j) \in D$, and the north color of $P_{(i,j)}$ equals the south color of $P_{(i,j+1)}$ whenever $(i,j), (i,j+1) \in D$.

Given a Turing machine $T$ with state set $Q$, tape alphabet $\Sigma$ and transition function $\delta$, one can construct a Wang tile set $S_T$ that simulates its computations~\cite[§4]{Ro71}.
The tiles are shown in Figure~\ref{fig:wang-tiles}, from left to right, for all valid combinations of the symbols and states:
\begin{enumerate}
\item
  a tape cell with symbol $s \in \Sigma$ and no head,
\item
  a tape cell with symbol $s$ and head in state $p \in Q$ with $\delta(p,s) = (q, r, {+1})$,
\item
  a tape cell with symbol $s$ and head entering from the left in state $q$,
\item
  a tape cell with symbol $s$ and head entering from the right in state $q$,
\item
  a tape cell with symbol $s$ and head in state $p$ with $\delta(p,s) = (q, r, {-1})$.
\end{enumerate}
It is clear that $T$ eventually halts if and only if there exists a valid rectangular pattern $P \in S_T^{[0, n-1] \times [0, m-1]}$ such that $P_{(0,0)}$ is a type-2 tile with $s = B$ and $p = q_0$, $P_{(0,m-1)}$ is a type-4 tile with $q = q_h$, $P_{(i,0)}$ is a type-1 tile with $s = B$ for all $1 \leq i < n$, and the west and east borders of $P$ are colored by the unmarked color.
We say that such a $P$ is a \emph{halting rectangle}.

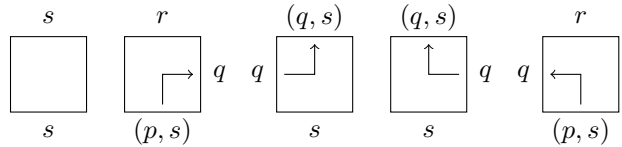
\begin{figure}[htp]
  \centering
  \begin{tikzpicture}

    \draw (0,0) rectangle ++(1,1);
    \node at (0.5,-0.25) {$s$};
    \node at (0.5,1.25) {$s$};

    \draw (1.5,0) rectangle ++(1,1);
    \node at (2,-0.25) {$(p, s)$};
    \node at (2,1.25) {$r$};
    \node at (2.75,0.5) {$q$};
    \draw [->] (2,0.1) -- (2,0.5) -- (2.4,0.5);

    \draw (3.5,0) rectangle ++(1,1);
    \node at (4,-0.25) {$s$};
    \node at (4,1.25) {$(q,s)$};
    \node at (3.25,0.5) {$q$};
    \draw [->] (3.6,0.5) -- (4,0.5) -- (4,0.9);

    \draw (5,0) rectangle ++(1,1);
    \node at (5.5,-0.25) {$s$};
    \node at (5.5,1.25) {$(q,s)$};
    \node at (6.25,0.5) {$q$};
    \draw [->] (5.9,0.5) -- (5.5,0.5) -- (5.5,0.9);

    \draw (7,0) rectangle ++(1,1);
    \node at (7.5,-0.25) {$(p, s)$};
    \node at (7.5,1.25) {$r$};
    \node at (6.75,0.5) {$q$};
    \draw [->] (7.5,0.1) -- (7.5,0.5) -- (7.1,0.5);

  \end{tikzpicture}
  \caption{The tile set $S_T$.}
  \label{fig:wang-tiles}
\end{figure}

\begin{proof}[Proof of Theorem~\ref{thm:undecidable}]
  We prove the two claims simultaneously by a reduction from the halting problem of Turing machines: given a Turing machine $T$, we construct (via a computable method) a binary freezing CA $g_T$ such that if $T$ halts on the empty input, then $g_T$ does not trivialize any full-support Bernoulli measure, and otherwise it trivializes them all.
  We first construct an auxiliary freezing CA $f_T$ on a finite poset alphabet $P$ that depends on $T$, and then show how to modify the construction to use the binary alphabet.
  The dynamics of $f_T$ is simply to check some local conditions (in symbolic dynamical terms, whether the configuration belongs to some subshift of finite type), turn any cell into a maximal state $M \in P$ when a local error is detected, and to propagate the state $M$ to neighboring cells according to certain conditions.

  Denote $D = \{N, NE, E, SE, S, SW, W, NW\}$.
  We interpret it as a set of cardinal and diagonal directions (north, north-east, east, etc).
  The state set of $f_T$ is ${P = S_T \cup \{M\}\cup D}$, where $S_T$ is the Wang tile set simulating $T$. We call a \emph{frame} any rectangular pattern whose perimeter is made with states from $D$ in the following way: the north, east, south, west  sides are respectively in state $N$, $E$, $S$, $W$, and the north-east, south-east, south-west, north-west corners are respectively in states $NE$, $SE$, $SW$, $NW$. We say that a frame contains a valid halting computation if its interior consists of $S_T$-cells that form a halting rectangle of $T$, as defined above.
  See Figure~\ref{fig:frame}.
  Then, valid configurations are those made of frames containing a valid halting computation, and with only state $M$ outside frames. All these conditions can be defined by a set of forbidden ${2\times 2}$ patterns (intuitively, the ${2\times 2}$ patterns not appearing in Figure~\ref{fig:frame}).

\begin{figure}[htp]
  \centering
  \begin{tikzpicture}[scale=0.8]
    \stR{0}{0}{$M$}; \stR{1}{1}{$M$}; \stR{1}{9}{$M$}; \stR{0}{10}{$M$}; \stR{9}{10}{$M$}; \stR{8}{9}{$M$}; \stR{8}{1}{$M$}; \stR{9}{0}{$M$};
    \draw[dotted] (.5,1)--(.5,10); \draw[dotted] (1.5,2)--(1.5,9); \draw[dotted] (9.5,1)--(9.5,10); \draw[dotted] (8.5,2)--(8.5,9);
    \draw[dotted] (1,.5)--(9,.5); \draw[dotted] (2,1.5)--(8,1.5); \draw[dotted] (1,10.5)--(9,10.5); \draw[dotted] (2,9.5)--(8,9.5);
    \stR{2}{2}{$SW$}; \stR{3}{2}{$S$}; \stR{6}{2}{$S$}; \stR{7}{2}{$SE$}; 
    \stR{2}{3}{$W$};  \stR{7}{3}{$E$}; \stR{2}{7}{$W$};  \stR{7}{7}{$E$};  
    \stR{2}{8}{$NW$}; \stR{3}{8}{$N$}; \stR{6}{8}{$N$}; \stR{7}{8}{$NE$}; 
    \draw[dotted] (2.5,4)--(2.5,7); \draw[dotted] (7.5,4)--(7.5,7);
    \draw[dotted] (4,2.5)--(6,2.5); \draw[dotted] (4,8.5)--(6,8.5);
    \draw[very thick] (2,2)--(2,9)--(8,9)--(8,2)--cycle;
    \draw[very thick] (3,3)--(3,8)--(7,8)--(7,3)--cycle;
    \stR{3}{3}{$q_0$}; \stR{3}{7}{$q_h$};
    \draw[->,very thick] (4,3.5)--(4.5,3.5)--(4.5,4)--(5.5,4.33)--(5.5,4.66)--(3.5,5)--(3.5,5.5)-- node[midway,sloped,above] {read-write head}(6.5,6)--(6.5,6.5)--(4.5,7)--(4.5,7.5)--(4,7.5);
  \end{tikzpicture}
  \caption{A frame encoding a valid halting computation.}
  \label{fig:frame}
\end{figure}

Depending on its state, we define the \emph{active neighbors} of a cell:
\begin{itemize}
\item a cell in a state from $P \setminus D$ has four active neighbors, one in each cardinal direction;
\item a cell in a state $d \in D$ has active neighbors in the cardinal directions which do not appear in $d$: \emph{e.g.} a cell in state $NE$ has active neighbours to the south and west.
\end{itemize}
The dynamics of $f_T$ is precisely the following:
\begin{enumerate}
\item if a cell belongs to some forbidden $2\times 2$ pattern, then it becomes $M$;
\item if a cell has an occurrence of $M$ among its active neighbors, then it becomes $M$;
\item otherwise the state doesn't change.
\end{enumerate}
The CA $f_T$ has neighborhood $\{-1,0,1\}^2$ and its local rule can be algorithmically determined from $T$.

From this definition, it is clear that when $T$ halts starting from the empty tape, then $f_T$ admits a finite obstacle pattern: a frame encoding a halting computation surrounded by state $M$ like in Figure~\ref{fig:frame}.
By Lemma~\ref{lem:CofiniteIce}, $f_T$ does not trivialize any nontrivial product measure.

Now suppose that $T$ does not halt and consider any configuration $x$ with an occurrence of $M$ on each semi-axis, both for positive and negative coordinates.
The set of such configurations has full measure under any full-support product measure.
We want to show that $f_T^n(x)_{\vec 0} = M$ for some $n \geq 0$. Suppose that $x_{\vec 0} \neq M$ and consider the (finite) maximal rectangle $R\subseteq\Z^2$ whose interior contains the origin but no occurrence of state $M$ in $x$. We proceed by induction on the size of $R$.
If $R$ is $1 \times 1$, then the origin is surrounded by $M$-cells and becomes $M$ in one step.

Suppose then that the claim holds for all strictly smaller rectangles.
Each side of $R$ has a neighbor outside of $R$ in state $M$ (if not $R$ would not be maximal). Hence, if $R$ is not a valid finite frame encoding a halting configuration then in one step at least one cell in the interior of $R$ becomes $M$ (either because the local condition is violated somewhere or because some cell has an active neighbor in state $M$). Since $T$ does not halt, this means that some $M$ appears in the interior of $R$, and in $f_T(x)$ we either have the cell at the origin in state $M$, or a smaller maximal rectangle of non-$M$ states so we can conclude by induction.

We now construct a binary freezing CA $g_T$ based on $f_T$ that shares its trivialization properties. We use a standard block encoding. Let $N$ be large enough to recode any state of $f_T$ as a ${N\times N}$ block of $0$s and $1$s in the following way:
\begin{itemize}
\item $M$ is coded by a ${N\times N}$ block of $1$s;
\item any state in ${P\setminus\{M\}}$ is coded by an ${N\times N}$ block made of an outer ${N\times N}$ annulus of $1$s, an inner ${(N-2)\times(N-2)}$ annulus of $0$s, and inside them a uniquely defined ${(N-4)\times(N-4)}$ pattern of $0$s and $1$s.
\end{itemize}
A given ${N\times N}$ block over alphabet ${\{0,1\}}$ is called valid if it is one of the coding blocks above, and invalid otherwise. The dynamics of $g_T$ is the following:
\begin{itemize}
\item if a cell is not inside the central block of some $3N \times 3N$ pattern of nine valid blocks, then it turns into $1$;
\item otherwise, if the local rule of $f_T$ applied to the pattern $w \in P^{3 \times 3}$ encoded by the blocks yields $M$, then it turns into $1$;
\item otherwise the cell retains its state.
\end{itemize}
Note that by choice of the coding (frame of $0$s inside a frame of $1$s) there is always at most one way to find a valid block around a cell which is in state $0$, hence the second case of the dynamics above is well-defined. Moreover, if a cell in state $0$ inside a valid block turns into $1$, then the entire block turns into $1^{N \times N}$. Finally, by construction, on properly encoded configurations, $g_T$ exactly simulates $f_T$.

Therefore, if $T$ halts, then the block encoding of a valid frame of $f_T$ (containing a halting computation) surrounded by ${N\times N}$ blocks of $1$s clearly forms a finite obstacle pattern under the dynamics of $g_T$. Now suppose that $T$ does not halt and consider a configuration $x$ such that the maximal rectangle $R$ around the origin not containing any ${N\times N}$ block of $1$s is finite.
Again, such configurations have full measure for any full support Bernoulli measure.
Like for $f_T$ we prove by induction on the size of $R$ that $g_T^n(x)_{\vec 0} = 1$ for some $n \geq 0$. If $R$ doesn't contain any valid block then every cell it contains turns into $1$ in one step and we are done. If $R$ contains a valid block with an invalid neighborhood, it turns into a ${N\times N}$ block of $1$s in one step and we can apply the induction hypothesis. If $R$ is entirely made of valid blocks with valid neighborhoods of blocks, we can apply the analysis of $f_T$ and show that $g_T(x)$ has a smaller maximal rectangle.
\end{proof}

We note that in terms of the arithmetical hierarchy, the third condition in Lemma~\ref{lem:CofiniteIce} is $\Pi^0_1$, so the previous proof shows that the set of binary freezing CA that trivialize some nontrivial full-support Bernoulli measure is $\Pi^0_1$-complete.

\section{Two phase transitions in freezing CA}

In this section we exhibit a freezing CA which has two phase transitions.

We need percolation results for measures that are not quite independent, but do not have long-range dependencies. One way to do this, which we opt for here, is to couple measures with independent ones and then use percolation results for independent distributions. For this we use a general result from \cite{LiScSt97}. 

\begin{definition}
  For a map ${\phi:\{0,1\}^{\Z^2}\rightarrow \{0,1\}^{\Z^2}}$, the \emph{dependence neighborhood} ${N_\phi(\vec z)}$ at position ${\vec z\in\Z^2}$ is the minimal set ${N\subseteq \Z^2}$ such that ${\phi(x)_{\vec z}}$ is determined by $x_N$, i.e. ${x_N=y_N}$ implies ${\phi(x)_{\vec z}=\phi(y)_{\vec z}}$. We say $\phi$ is \emph{$k$-dependent} (${k\in\N}$) if for any $\vec z, \vec z'$ with ${\|\vec z-\vec z'\|_\infty>k}$ it holds ${N_\phi(\vec z)\cap N_\phi(\vec z')=\emptyset}$.
\end{definition}

\begin{definition}
  \label{def:pathsets}
  For $b \in \{0,1\}$ and $\vec z \in \Z^2$, denote by $P^+_b(\vec z)$ the set of configurations $x \in \{0,1\}^{\Z^2}$ such that there is an infinite path $(a_i)_{i \in \N}$ in $\Z^2$ such that $a_0 = \vec z$, $a_{i+1} - a_i \in \{(0,1), (1,1)\}$ and $x_{a_i} = b$ for all $i \geq 0$.
  Denote its complement by $P^-_b(\vec z)$.
  We may abbreviate $P^+_b(\vec z) = P^+_b$ and $P^-_b(\vec z) = P^-_b$ when $\vec z$ is clear from the context.
\end{definition}

\begin{lemma}
  \label{lem:lowerbernoulli}
  For any ${k\in\N}$, there is ${0<p_k<1}$ such that for any ${0<p\leq 1}$ and for any $k$-dependent map ${\phi:\{0,1\}^{\Z^2}\rightarrow \{0,1\}^{\Z^2}}$ verifying ${\mu_p(\phi^{-1}([1]_{\vec z}))>p_k}$ for all ${\vec z\in\Z^2}$ it holds:
  \begin{itemize}
  \item ${\mu_p(\phi^{-1}(P_0^-(\vec z)))=1}$
  \item ${\mu_p(\phi^{-1}(P_1^+(\vec z)))>0}$
  \end{itemize}
\end{lemma}
\begin{proof}
  Let us first consider the case where $\phi$ is the identity map.
  Let $p_c = \sup\{p : \mu_p(P_1^+)=0\}$, which is the critical probability of directed percolation on the square lattice.
  It is a classical result from percolation theory that $0 < p_c < 1$.
  By symmetry between states $0$ and $1$, for ${p>\max(p_c,1-p_c)}$ both ${\mu_p(P_0^-)=1}$ and ${\mu_p(P_1^+)>0}$ hold.

  The general case follows from the results of~\cite{LiScSt97}. Simplifying the setting a bit to match our needs, we say a measure $\mu \in \meas(\{0,1\}^{\Z^2})$ is $k$-dependent if for any ${A,B\subseteq\Z^2}$ with ${\min \{ \|\vec z_A-\vec z_B\|_\infty \;|\; \vec z_A \in A, \vec z_B \in B \} > k}$, the random variables $x|_A$ and $x|_B$ are independent when $x \in \{0,1\}^{\Z^2}$ is drawn from $\mu$. On the other hand, we say $\mu$ dominates another measure $\mu' \in \meas(\{0,1\}^{\Z^2})$ if for any upper-closed measurable set $E$ (\textit{i.e.} ${x\in E}$ and ${x_{\vec z}\leq y_{\vec z}}$ for all $\vec z\in\Z^2$ implies $y\in E$) we have ${\mu(E)\geq\mu'(E)}$. Then~\cite[Theorem 0.0]{LiScSt97} implies that for any $k \geq 0$ and ${0<p<1}$, if $\mu \in \meas(\{0,1\}^{\Z^2})$ is a $k$-dependent measure and $\min \{ \mu([1]_{\vec z}) \;|\; \vec z \in \Z^2 \}<1$ is large enough, then $\mu$ dominates the product measure $\mu_p$.
  Note that $\mu$ need not be shift-invariant.

  For any $k$-dependent map $\phi$ and any product measure $\mu_p$ it is the case that $\mu_p\circ\phi^{-1}$ is $k$-dependent. The domination result above allows to conclude since both $P_0^-$ and $P_1^ +$ are upper-closed sets.
\end{proof}

\begin{theorem}
\label{thm:weird-nilpotent}
There exists a freezing CA $f$ and Bernoulli measures $\mu_{\epsilon_1}$ and $\mu_{\epsilon_2}$ and $\mu_{\epsilon_3}$ on $\{0,1\}^{\Z^2}$ such that $0<\epsilon_1 < \epsilon_2< \epsilon_3<1$, and $f$ trivializes $\mu_{\epsilon_1}$ and $\mu_{\epsilon_3}$ but not $\mu_{\epsilon_2}$.
\end{theorem}

The CA of the above theorem can only turn 0s into 1s, and starting from a $\mu_{\epsilon_1}$-random configuration it will converge towards the all-$1$ configuration, but starting from a $\mu_{\epsilon_2}$-random configuration, which has a strictly higher density of 1s, it will leave some cells in state 0 forever.
To understand the fundamental use of non-monotony in the construction and resolve this apparent paradox, the basic idea is the following: If a freezing CA has a large enough neighborhood, it can locally ``see'' a good enough approximation of the density of 1s in the initial configuration.
Then, if this local estimate of density is close to $\epsilon_1$ it can produce a lot of 1s locally, while if it is close to $\epsilon_2$ it doesn't. 

This density jump when starting from $\epsilon_1$ is chosen so that it passes the critical probability of some percolation process (a modification of Example~\ref{ex:simpleca}), while the system stays below this critical probability when starting close enough to $\epsilon_2$.
There are thus two processes going on: a density modification and a percolation process.
Our main trick is to use a block encoding to avoid interactions between the two.
A block of 1s encodes a single cell in state $1$ of the percolation process of Example~\ref{ex:simpleca}, and, when such a block appears in the neighborhood, the density modification is inhibited.
Thanks to this trick, the CA behaves as if the density modification was only applied once at the initial step, and then successive steps just reproduce the percolation process on a modified initial configuration.
Note that the existence of $\epsilon_3$ is granted by Lemma~\ref{lem:CofiniteIce} just because $f$ does trivialize some full support measure.

Recall Hoeffding's inequality~\cite{Ho63}, which will be used to quantify the local estimated density of 1s: if $\mathbb{P}(n,p,\epsilon)$ is the probability that the average of $n$ binary Bernoulli trials, where the probability of $1$ is $p$, does not lie in $[p-\epsilon, p+\epsilon]$, then
\begin{equation}
  \label{eq:technicalbernoulli}
  \mathbb{P}(n,p,\epsilon) \leq 2 \exp \left( -2 \epsilon^2 n \right )
\end{equation}
for all $\epsilon > 0$.

\begin{proof}[Proof of Theorem~\ref{thm:weird-nilpotent}]
We first define the automaton $f$ using some undefined parameters, and then show that it is correct for some choice of said parameters. The parameters are $N \in \N$ and $\epsilon_1, \epsilon_2, \delta \in \R$ subject to $0 < \delta < \epsilon_1 < \epsilon_2 - \delta$ and $\epsilon_2 < 1$, so that $f = f_{N, \epsilon_1, \epsilon_2, \delta}$. We will define it in two phases, such that $f = h \circ g$ for some other CA $g$ and $h$.

We first define an auxiliary CA $g'$ by
\[
  g'(x)_{\vec 0} =
  \begin{cases}
    1, & \text{if $x_{\vec z + [0,N-1]^2} \not\equiv 1$ for all $\vec z \in [-4N,3N]^2$} \\
    & \text{and $|x_{[-N,N]^2}|_1 / (2N+1)^2 \in (\epsilon_1 - \delta, \epsilon_1 + \delta)$,} \\
    0, & \text{otherwise.}
  \end{cases}
\]
The role of $g'$ is to indicate regions in which the density modification process should take place.
The CA $g$, which implements said process, is then defined by
\[
  g(x)_{\vec 0} =
  \begin{cases}
    1, & \text{if $x_{\vec 0} = 1$ or $g'(x)_{\vec z + [0,N-1]^2} \equiv 1$} \\
    & \text{for some $\vec z \in [-N+1,0]^2$,} \\
    0, & \text{otherwise.}
  \end{cases}
\]
The CA $h$, which implements the percolation process, is defined by
\[
  h(x)_{\vec 0} =
  \begin{cases}
    1, & \text{if $x_{\vec 0} = 1$ or $x_{\vec z + [0,2N-1] \times [N,2N-1]} \equiv 1$} \\
    & \text{for some $\vec z \in [-N+1,0]^2$,} \\
    0, & \text{otherwise.}
  \end{cases}
\]

\begin{claim}
  \label{cl:idemp-g}
  For all $n \geq 0$ we have $g \circ h^n \circ g = h^n \circ g$.
\end{claim}

In particular, $f^n = h^n \circ g$ for any integer ${n\geq 1}$.

\begin{proof}
  Say that a configuration $x \in \{0,1\}^{\Z^2}$ is \emph{nice} if for all $\vec z \in \Z^2$ such that $x_{\vec z} = 0$, we have $g'(x)_{\vec v + [0,N-1]^2} \not\equiv 1$ for all $\vec v \in \vec z - [0,N-1]^2$. It is now enough to prove for all $x \in \{0,1\}^{\Z^2}$ that
  \begin{enumerate}
  \item $g(x)$ is nice,
  \item if $x$ is nice, then $x = g(x)$, and
  \item if $x$ is nice, then $h(x)$ is nice.
  \end{enumerate}
  The second item is clear by the definition of $g$. We prove the first and third items using the fact that $g$ and $h$ only change cells by creating large batches of $1$s, which suppress the density modification process.
  The proof is illustrated in Figure~\ref{fig:idemp-g}.

  For the first item, suppose for a contradiction that $g(x)$ is not nice: for some $\vec z \in \Z^2$ and $\vec z_1 \in \vec z - [0, N-1]^2$ we have $g(x)_{\vec z} = 0$ and $g'(g(x))_{\vec z_1 + [0,N-1]^2} \equiv 1$ (the higher dashed square in Figure~\ref{fig:idemp-g}).
  By the definition of $g$, we in particular have $g'(x)_{\vec z_1 + [0,N-1]^2} \not\equiv 1$, so there is some $\vec z_2 \in \vec z_1 + [0, N-1]^2$ with $g'(x)_{\vec z_2} = 0$ and $g'(g(x))_{\vec z_2} = 1$.
  The latter equation implies
  \begin{equation}
    \label{eq:no-blobs}
    g(x)_{\vec v + [0,N-1]^2} \not\equiv 1
  \end{equation}
  for all $\vec v \in \vec z_2 + [-4N, 3N]^2$, and since $g$ is freezing, the same holds for $x$ in place of $g(x)$.
  Thus the density of 1s in $g(x)_{\vec z_2 + [-N,N]^2}$ (the solid square in Figure~\ref{fig:idemp-g}) lies in $(\epsilon_1-\delta, \epsilon_1+\delta)$, but that of $x_{\vec z_2 + [-N,N]^2}$ does not.
  Hence $g(x)_{\vec z_3} \neq x_{\vec z_3}$ for some $\vec z_3 \in \vec z_2 + [-N, N]^2$.
  This, in turn, implies that there is some ${\vec z_4 \in \vec z_3+[-N+1,0]^2}$ such that ${g'(x)_{\vec z_4+[0,N-1]^2} \equiv 1}$ (the lower dashed square in Figure~\ref{fig:idemp-g}).
  Now we in fact have ${g(x)_{\vec z_4 + [0, N-1]^2} \equiv 1}$.
  This contradicts~\eqref{eq:no-blobs} since $\vec z_4 \in \vec z_2 + [-2N+1, N]^2$.

  For the third item, suppose for a contradiction that $x$ is nice but $h(x)$ is not: for some $\vec z \in \Z^2$ and $\vec z_1 \in \vec z - [0, N-1]^2$ we have $h(x)_{\vec z} = 0$ and $g'(h(x))_{\vec z_1 + [0,N-1]^2} \equiv 1$.
  Then $x_{\vec z} = 0$, and by niceness ${g'(x)_{\vec z_1+[0,N-1]^2} \not\equiv 1}$.
  Take $\vec z_2 \in \vec z_1+[0,N-1]^2$ with $g'(x)_{\vec z_2} = 0$.
  As in the previous paragraph, we have
  \begin{equation}
    \label{eq:no-blobs2}
    h(x)_{\vec v + [0,N-1]^2} \not\equiv 1
  \end{equation}
  for all $\vec v \in \vec z_2 + [-4N, 3N]^2$ and similarly for $x$, and $h(x)_{\vec z_3} \neq x_{\vec z_3}$ for some $\vec z_3 \in \vec z_2 + [-N, N]^2$.
  By the definition of $h$, there now exists ${\vec z_4 \in \vec z_3+[-N+1,0]^2}$ with ${x_{\vec z_4 + [0,2N-1]\times [N,2N-1]} \equiv 1}$ (the dashed rectangle in Figure~\ref{fig:idemp-g}).
  By choosing ${\vec v=\vec z_4 + (0,N)}$ and since $h$ is freezing this yields ${h(x)_{\vec v + [0,N-1]^2}\equiv 1}$.
  However $\vec v \in \vec z_2 + [-2N+1, 2N]^2$, contradicting \eqref{eq:no-blobs2}.
\end{proof}

\begin{figure}[htp]
  \centering
  \begin{tikzpicture}[scale=2]

    \fill[black] (0,0) circle (0.03cm);
    \node [right] at (0,0) {$\vec{z}$};

    \fill[black] (-0.5,-0.75) circle (0.03cm);
    \node [below] at (-0.5,-0.75) {$\vec{z}_1$};
    \draw [dashed] (-0.5,-0.75) rectangle ++(1,1);

    \fill[black] (0.25,-0.5) circle (0.03cm);
    \node [left] at (0.25,-0.5) {$\vec{z}_2$};
    \draw (-0.75,-1.5) rectangle ++(2,2);

    \fill[black] (1.1,0.35) circle (0.03cm);
    \node [left] at (1.1,0.35) {$\vec{z}_3$};

    \fill[black] (0.75,-0.2) circle (0.03cm);
    \node [below] at (0.75,-0.2) {$\vec{z}_4$};
    \draw [dashed] (0.75,-0.2) rectangle ++(1,1);
    \draw [dashed] (0.75,0.8) -- ++(0,1) -- ++(2,0) -- ++(0,-1) -- ++(-1,0);

    \fill[black] (0.75,0.8) circle (0.03cm);
    \node [left] at (0.75,0.8) {$\vec{v}$};

  \end{tikzpicture}
  \caption{An illustration of the proof of Claim~\ref{cl:idemp-g}.}
  \label{fig:idemp-g}
\end{figure}
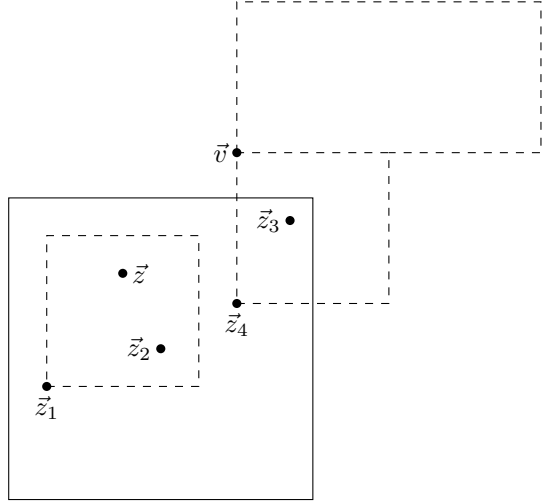

Define two functions $A$ and $B$ from $\{0,1\}^{\Z^2}$ to itself by
\begin{align*}
  A(x)_{(a,b)} & = \begin{cases}
    1, & \text{if $g(x)_{(aN, bN)} = 0$,} \\
    0, & \text{otherwise,}
  \end{cases} \\
  B(x)_{(a,b)} & = \begin{cases}
    1, & \text{if $g(x)_{(aN, bN) + [0,N-1]^2} \equiv 1$,} \\
    0, & \text{otherwise.}
  \end{cases}
\end{align*}
Since the radius of $g$ is at most $5N$, the two functions $A$ and $B$, when seen as random variables over a Bernoulli measure on $\{0,1\}^{\Z^2}$, are $5$-dependent as defined in Lemma~\ref{lem:lowerbernoulli}. Following the notations of Definition~\ref{def:pathsets} we make the following claim.

\begin{claim}
  \label{cl:paths}
  Let $x \in \{0,1\}^{\Z^2}$.
  \begin{itemize}
  \item If ${A(x) \in P_1^+(\vec 0)}$ then ${f^t(x)_{\vec 0}=0}$ for all $t$. 
  \item If ${B(x) \in P_0^-(\vec 0)}$ then ${f^t(x)_{\vec 0}=1}$ for some $t$.
  \end{itemize}
\end{claim}

\begin{proof}
  If ${A(x) \in P_1^+}$ then there is an infinite path ${(\vec z_i)_{i\in\N}}$ as in Definition~\ref{def:pathsets} with ${A(x)_{\vec z_i}=1}$ for all ${i\in\N}$ and $\vec z_0 = \vec 0$. Let us show by induction on $t$ that ${h^t(g(x))_{N\vec z_i}=0}$ for all ${i\in\N}$. The first item follows from this by Claim~\ref{cl:idemp-g}.
  By hypothesis and by definition of $A$ it is true for ${t=0}$. Suppose now that it holds for some $t \geq 0$.
  For any ${i\in\N}$ we have ${N\vec z_{i+1}\in N\vec z_i + \{(0,N),(N,N)\}}$.
  By hypothesis ${h^t(g(x))_{N\vec z_i}=h^t(g(x))_{N\vec z_{i+1}}=0}$, so we deduce ${h^{t+1}(g(x))_{N\vec z_i}=0}$ by the definition of $h$.

  Suppose that ${B(x) \in P_0^-}$ and consider finite paths $(\vec z_i)_{0 \leq i \leq m}$ with $\vec z_0 = \vec 0$ and $\vec z_{i+1} - \vec z_i \in \{(0,1), (1,1)\}$ for each $0 \leq i < m$.
  By K\H{o}nig's lemma there is a bound $\beta(x)$ on the length of those paths with ${B(x)_{\vec z_i}=0}$ for all ${0\leq i\leq m}$. By definition of $B$ and $h$, we have $h(g(x)) \in P_0^-$ and $\beta(h(g(x))) \leq \beta(x)-1$ (recall that $g(h(g(x))) = h(g(x))$ by Claim~\ref{cl:idemp-g}). By immediate induction we have ${f^{\beta(x)}(x)_{\vec 0}=h^{\beta(x)}(g(x))_{\vec 0}=1}$ and the second item of the claim follows.
\end{proof}

In order to apply Lemma~\ref{lem:lowerbernoulli} simultaneously to $A$ and $B$, we need to choose the parameters of our construction so that the marginals of both ${\mu_{\epsilon_2} A^{-1}([1])}$ and ${\mu_{\epsilon_1} B^{-1}([1])}$ are close enough to $1$. We claim that it is possible.

\begin{claim}
  \label{cl:params}
  For any $0<\epsilon<1$ there are parameters $N$, $\epsilon_1$, $\epsilon_2$ and $\delta$ such that we have simultaneously ${\mu_{\epsilon_2} A^{-1}([1]_{\vec z}) > 1-\epsilon}$ and ${\mu_{\epsilon_1} B^{-1}([1]_{\vec z})>1-\epsilon}$ for all ${\vec z\in\Z^2}$.
\end{claim}

\begin{proof}
  It is sufficient to prove the inequalities for ${\vec z=\vec 0}$ since Bernoulli measures are translation invariant and translations are turned into translations through $A^{-1}$ and $B^{-1}$.
  We first show that $\epsilon_1$, $\epsilon_2$ and $\delta$ can be chosen so that for all $N$ large enough we have ${\mu_{\epsilon_2} A^{-1}([1]_{\vec 0}) > 1-\epsilon}$. By definition $A(x)_{\vec 0}=1$ iff $g(x)_{\vec 0}=0$. So ${\mu_{\epsilon_2} A^{-1}([1]_{\vec 0}) \geq 1 - \mu_{\epsilon_2} g^{-1}([1]_{\vec 0})}$. Moreover, by the definition of $g$ and $g'$ we have $g^{-1}([1]_{\vec 0}) \subseteq [1]_{\vec 0} \cup D_{\vec 0}$, where $D_{\vec z}$ is the set of configurations $x \in \{0,1\}^{\Z^2}$ such that the finite pattern $x_{\vec z+[-N,N]^2}$ has a density of $1$-symbols strictly between $\epsilon_1-\delta$ and $\epsilon_1+\delta$. We deduce that
\[
  \mu_{\epsilon_2} A^{-1}([1]_{\vec 0})\geq 1 - \epsilon_2 - \mu_{\epsilon_2}(D_{\vec 0}).
\]
Let us fix $\epsilon_2 = \epsilon/2$ and any values of $\epsilon_1$ and $\delta$ such that $0 < \delta < \epsilon_1 < \epsilon_2 - \delta$. From~\eqref{eq:technicalbernoulli} we know that ${\mu_{\epsilon_2}(D_{\vec 0})} \longrightarrow 0$ as $N$ grows. We deduce that for large enough $N$, we have ${\mu_{\epsilon_2} A^{-1}([1]_{\vec 0}) > 1-\epsilon}$.

Let us now prove that for the parameters $\epsilon_1$, $\epsilon_2$ and $\delta$ fixed above and for large enough $N$, we also have ${\mu_B([1]_0)>1-\epsilon}$.
We compute
\begin{align*}
  \mu_{\epsilon_1} B^{-1}([1]_{\vec 0}) \geq {} & \mu_{\epsilon_1} (\{x: g'(x)_{[0,N-1]^2} \equiv 1\})\\
  {} \geq {} & 1 - \sum_{\vec z \in [-4N, 4N-1]^2} \mu_{\epsilon_1} (\{x : x_{\vec z+[0,N-1]^2} \equiv 1\})\\
                                                & - \sum_{\vec z \in [0, N-1]} \mu_{\epsilon_1} (\{0,1\}^{\Z^2} \setminus D_{\vec z})\\
  {} \geq {} & 1 - 81 \cdot N^2 \cdot \epsilon_1^{N^2} - N^2 (1-\mu_{\epsilon_1} (D_{\vec 0})).
\end{align*}
The last expression goes to $1$ as $N$ goes to infinity because by~\eqref{eq:technicalbernoulli}, $\mu_{\epsilon_1} (D_{\vec 0}) \longrightarrow 1$ exponentially fast as $N$ grows.
\end{proof}

From Claim~\ref{cl:paths}, Claim~\ref{cl:params} and Lemma~\ref{lem:lowerbernoulli} we deduce that $f$ trivializes $\mu_{\epsilon_1}$ but not $\mu_{\epsilon_2}$.
The existence of $\epsilon_3$ follows from Lemma~\ref{lem:CofiniteIce} since $f$ does trivialize some full support measure.
\end{proof}

\section{Future directions}
\label{sec:future}

A natural (at least for cellular automata theorists) generalization of bootstrap percolation would be to consider the trivialization properties of freezing monotone CA on arbitrary poset alphabets.
To our knowledge, the following question is open even in the case of $P = \{0,1,2\}$ with the linear ordering.

\begin{question}
  Given a finite poset $P$ and a freezing monotone CA $f$ on $P^{\Z^2}$, is it decidable whether
  \begin{enumerate}
  \item
    $f$ trivializes all full-support product measures?
  \item
    $f$ trivializes some full-support product measure?
  \end{enumerate}
\end{question}

While we show in Section~\ref{sec:non-monotone} that the set of binary freezing CA trivializing at least one nontrivial Bernoulli measure is $\Pi^0_1$-complete, we have not been able to pinpoint the complexity of those CA that trivialize them all.
In fact, it is not immediately clear whether this set is even arithmetical.

\begin{question}
  What is the complexity of the set of binary freezing CA that trivialize all nontrivial Bernoulli measures?
\end{question}

Theorem~\ref{thm:weird-nilpotent} shows that the property of trivializing a Bernoulli measure need not be monotone with respect to the measure for a fixed binary freezing CA.
We believe our construction only scratches the surface of the measure trivialization property, and that much more intricate constructions are possible.
More explicitly, for a CA $f$ on $\{0,1\}^{\Z^2}$, let ${T(f) = \{ 0 \leq p \leq 1 \;|\; \text{$f$ trivializes $\mu_p$} \}}$.
If $f$ is freezing, then $1 \in T(f)$, and Theorem~\ref{thm:weird-nilpotent} shows that in this case $T(f) \setminus \{0\}$ need not be an interval.
Apart from these facts, it is not clear to us how intricate the structure of $T(f)$ can be.

\begin{question}
  What is the class of sets $T(f)$ for freezing CA $f$ on $\{0,1\}^{\Z^2}$?
\end{question}

\section*{Acknowledgments}

We are extremely grateful to an anonymous referee for their critique of our original draft, which led to substantial changes in exposition.

\printbibliography

\end{document}